\documentclass[12pt]{amsart}
\usepackage{bbm}
\usepackage[pdftex]{graphicx}
\usepackage{amscd, amssymb, dsfont, upgreek, mathrsfs}
\input{xy}
\xyoption{all}

\newtheorem{Thm}{Theorem}
\newtheorem{Conj}[Thm]{Conjecture}

\newtheorem{Def}[Thm]{Definition}
\newtheorem{Def/Thm}[Thm]{Definition/Theorem}

\newtheorem{Lemma}[Thm]{Lemma}

\theoremstyle{remark}
\newtheorem{Rmk}[Thm]{Remark}

\newcommand{\PP }{{\mathbb P}}

\newcommand{\CC }{{\mathbb C}}
\newcommand{\ZZ }{{\mathbb Z}}

\begin{document}

\title[Note on equivariant $I$-function of local $\mathds{P}^n$]
{Note on equivariant $I$-function of local $\mathds{P}^n$}

\author{Hyenho Lho}
\address{Department of Mathematics, ETH Z\"urich}
\email {hyenho.lho@math.ethz.ch}
\date{July 2018}.

\begin{abstract} Several properties of a hyepergeometric series related to Gromov-Witten theory of some Calabi-Yau geometries was studied in \cite{ZaZi}. These properties play basic role in the study of higher genus Gromov-Witten theories. We extend the results of \cite{ZaZi} to equivariant setting for the study of higher genus equivariant Gromov-Witten theories of some Calabi-Yau geometries.

\end{abstract}

\maketitle

\setcounter{tocdepth}{1} 
\tableofcontents

\setcounter{section}{-1}

\section{Introduction}
\subsection{Local $\PP^n$ geometries}\label{LPPN}
Equivariant local $\PP^n$ theories can be constructed as follows. Let the algebraic torus
$$\mathsf{T}_{n+1}=(\CC^*)^{n+1}$$
act with the standard linearization on $\PP^n$ with weights $\lambda_0,\dots,\lambda_n$ on the vector space $H^0(\PP^n,\mathcal{O}_{\PP^n}(1))$.
 Let $\overline{M}_g(\PP^n,d)$ be the moduli space of stable maps to $\PP^n$ equipped with the canonical $\mathsf{T}_{n+1}$-action, and let
 $$\mathsf{C}\rightarrow\overline{M}_g(\PP^n,d)\,,\,\,f:\mathcal{C}\rightarrow\PP^n\,,\,\,\mathsf{S}=f^*\mathcal{O}_{\PP^n}(-1)\rightarrow\mathsf{C}$$
be the standard universal structures. 

The equivariant Gromov-Witten invariants of the local $\PP^n$ are defined via the equivariant integrals
\begin{align}\label{GW}
    N_{g,d}^{\mathsf{GW},n}=\int_{[\overline{M}_g(\PP^n,d)]^{\text{vir}}}e\Big(-R\pi_*f^* \mathcal{O}_{\PP^n}(-n-1)\Big)\,.
\end{align}
The integral \eqref{GW} defines a rational function in $\lambda_i$
$$N_{g,d}^{\mathsf{GW},n}\in \CC(\lambda_0,\dots,\lambda_n)\,.$$
We associate Gromov-Witten generating series by
$$\mathcal{F}^{\mathsf{GW},n}_g(Q)\, =\, 
\sum_{d=0}^\infty \widetilde{N}_{g,d}^{\mathsf{GW}} Q^d\, 
\in \, \CC(\lambda_0,\dots,\lambda_n)[[Q]] \, .$$
Motivated by mirror symmetry (\cite{ASYZ,BCOV,ObPix}), we can make the following predictions about the genus $g$ generating series $\mathcal{F}^{\mathsf{GW},n}_g$. 
\begin{itemize}
    \item[(A)] There exist a finitely generated subring 
    $$\mathds{G}\in\CC(\lambda_0,\dots,\lambda_n)[[Q]]$$
    which contains $\mathcal{F}_g^{\mathsf{GW},n}$ for all $g$.
    \item[(B)] The series $\mathcal{F}_g^{\mathsf{GW},n}$ satisfy {\em holomorphic anomaly equations}, i.e. recursive formulas for the derivative of $\mathcal{F}_g^{\mathsf{GW},n}$ with respect to some generators in $\mathds{G}$.
\end{itemize}

\subsection{$I$-function}
$I$-fucntion defined by
$$I_n=\sum_{d=0}^{\infty}\frac{\prod_{k=1}^{(n+1)d-1}(-(n+1)H-kz)}{\prod_{i=0}^n \prod_{k}^d(H+kz-\lambda_i)}q^d\in H_\mathsf{T}^*(\PP^n,\CC)[[q]]\,,$$
is the central object in the study of Gromov-Witten invariants of local $\PP^n$ geometry. See \cite{LP}, \cite{LP2} for the arguments. Several important properties of the function $I_n$ was studied in \cite{ZaZi} after the specialization
\begin{align}\label{sp}
    \lambda_i=\zeta_{n+1}^i\,
\end{align}
where $\zeta_{n+1}$ is primitive $(n+1)$-th root of unity. For the study of full equivariant Gromov-Witten theories, we extend the result of \cite{ZaZi} without the specialization \eqref{sp}.

\subsection{Picard-Fuchs equation and Birkhoff factorization}
Define differential operators
\begin{align*}
    \mathsf{D}=q \frac{d}{dq}\,,\,\,\,M=H+z\mathsf{D}\,.
\end{align*}
The function $I_n$ satisfies following Picard-Fuchs equation
\begin{align*}
    \Big(\prod_{i=0}^n\Big(M-\lambda_i\Big)-q\prod_{k=0}^{n}\Big(-(n+1)M-kz\Big)\Big)I_n=0\,.
\end{align*}

The restriction $I_n|_{H=\lambda_i}$ admits following asymptotic form
\begin{align}\label{asymp}
    I_n|_{H=\lambda_i}=e^{\frac{\mu}{z}}\Big( R_{0,i}+R_{1,i}z+R_{2,i}z^2+\dots \Big)
\end{align}
with series $\mu_i\,,\,R_{k,i}\in \CC(\lambda_0,\dots,\lambda_n)[[q]]$.

 A derivation of \eqref{asymp} is obtained from \cite[Theorem 5.4.1]{CKg0} and the uniqueness lemma in \cite[Section 7.7]{CKg0}. The series $\mu_i$ and $R_{k,i}$ are found by solving defferential equations obtained from the coefficient of $z^k$. For example,
 \begin{align*}\lambda_i+\mathsf{D}\mu_i= L_i\,,
 \end{align*}
where $L_i(q)$ is the series in $q$ defined by the root of following degree $(n+1)$ polynomial in $\mathcal{L}$
\begin{align*}
    \prod_{i=0}^{n}(\mathcal{L}-\lambda_i)-(-1)^{n+1}q \mathcal{L}^{n+1}\,.\,
\end{align*}
with initial conditions,
$$\mathcal{L}_i(0)=\lambda_i\,.$$

Let $f_n$ be the polynomial of degree $n$ in variable $x$ over $\CC(\lambda_0,\dots,\lambda_n)$ defined by
\begin{align*}
    f_n(x):=\sum_{k=0}^{n} (-1)^{k}k s_{k+1} x^{n-k}\,,
\end{align*}
where $s_k$ is $k$-th elementary symmetric function in $\lambda_0,\dots,\lambda_n$. The ring
\begin{align*}
    \mathds{G}_n:=\CC(\lambda_0,\dots,\lambda_n)[L_0^{\pm 1},\dots,L_n^{\pm 1},f_n(L_0)^{-\frac{1}{2}},\dots,f_n(L_n)^{-\frac{1}{2}}]
\end{align*}
will play a basic role. 

The following Conjecture was proven under the specializaiton \eqref{sp} in \cite[Theorem 4]{ZaZi}.

\begin{Conj}\label{MC}
 For all $k\ge 0$, we have
$$R_{k,i}\in \mathds{G}_n\,.$$
\end{Conj}

Conjecture \ref{MC} for the case $n=1$ will be proven in Section \ref{PP1}. Conjecture \ref{MC} for the case $n=2$ will be proven in Section \ref{PP2} under the specialization \eqref{spl2}.
In fact, the argument in Section \eqref{PP2} proves Conjecture \ref{MC} for all $n$ under the specialization which makes $f_n(x)$ into power of a linear polynomial.

\subsection{Acknowledgments}
I am grateful to  R. Pandharipande for useful discussions.
I was supported by the grant ERC-2012-AdG-320368-MCSK.

\section{Admissibility of differential equations}
Let $\mathsf{R}$ be a commutative ring. Fix a polynomial $f(x)\in \mathsf{R}[x]$. We consider a differential operator of {\em level} $n$ with following forms.

\begin{align}
    \mathcal{P}(A_{lp},f)[X_0,\dots,X_{n+1}]=\mathsf{D}X_{n+1}-\sum_{n\ge l\ge 0,\,p\ge 0}A_{lp} \mathsf{D}^p X_{n-l}\,,
\end{align}
where $\mathsf{D}:=\frac{d}{dx}$ and $A_{lp}\in\mathsf{R}[x]_f:= \mathsf{R}[x][f^{-1}]$. We assume that only finitely many $A_{lp}$ are not zero.

\begin{Def}\label{admissible}
Let $R_i$ be the solutions of the equations for $k\ge 0$,
\begin{align}\label{DO}
    \mathcal{P}(A_{lp},f)[X_{k+1},\dots,X_{k+n}]=0\,,
\end{align}
with $R_0=1$. We use the conventions $X_{i}=0$ for $i<0$.
We say differential equations \eqref{DO} is {\em admissible} if the solutions $R_k$ satisfies for $k \ge 0$,
$$R_k\in \mathsf{R}[x]_f\,.$$
\end{Def}

\begin{Rmk}
Note that the admissibility of $\mathcal{P}(A_{lp},f)$ in Definition \ref{admissible} do not depend on the choice of the solutions $R_k$.  
\end{Rmk}

\begin{Lemma}
Let $f$ be a degree one polynomial in $x$. Each $A\in\mathsf{R}[x]_f$ can be written uniquely as
$$A=\sum_{i\in\ZZ}a_i f^i\,$$
with finitely many non-zero $a_i\in\mathsf{R}$. We define {\em the order} $\text{Ord}(A)$ of $A$ with respect to $f$ by smallest $i$ such that $a_i$ is not zero. Then

\begin{align*}
    \mathcal{P}(A_{lp},f)[X_0,\dots,X_{n+1}]:=\mathsf{D}X_{n+1}-\sum_{n\ge l\ge 0,\,p\ge 0}A_{lp} \mathsf{D}^p X_{n-l}=0\,,
\end{align*}
is admissible if following condition holds:
\begin{align}\label{D1C}
   \nonumber \text{Ord}(A_{l0})&\le -2\,,\\
    \text{Ord}(A_{l1})&\le 0\,,\\
   \nonumber \text{Ord}(A_{lp})&\le p+1\,\,\,\,\text{for}\,\,p\ge2\,.
\end{align}
\end{Lemma}

\begin{proof}
The proof follows from simple induction argument.
\end{proof}

\begin{Lemma}
Let $f$ be a degree two polynomial in $x$. Denote by
$$\mathsf{R}_f$$
the subspace of $\mathsf{R}[x]_f$ generated by $f^{i}$ for $i\in \ZZ$. Each $A\in\mathsf{R}_f$ can be written uniquely as
$$A=\sum_{i\in\ZZ}a_i f^i\,$$
with finitely many non-zero $a_i\in\mathsf{R}$. We define {\em the order} $\text{Ord}(A)$ of $A\in\mathsf{R}_f$ with respect to $f$ by smallest $i$ such that $a_i$ is not zero. Then
\begin{align*}
    \mathcal{P}(A_{lp},f)[X_0,\dots,X_{n+1}]:=\mathsf{D}X_{n+1}-\sum_{n\ge l\ge 0,\,p\ge 0}A_{lp} \mathsf{D}^p X_{n-l}=0\,,
\end{align*}
is admissible if following condition holds:

\begin{align*}
    A_{lp}&=B_{lp}\,\,\,\,\,\,\,\,\,\,\,\,\,\,\,\,\,\,\text{if p is odd}\,,\\
    A_{lp}&=\frac{df}{dx}\cdot B_{lp}\,\,\,\,\,\text{if p is even}\,,
\end{align*}
where $B_{lp}$ are elements of $\mathsf{R}_f$ with
\begin{align*}
    \text{Ord}(B_{l0})&\le -2\,,\\
    \text{Ord}(B_{lp})&\le \Big[\frac{p-1}{2}\Big]\,\,\,\,\text{for}\,\,p\ge1\,.
\end{align*}

\end{Lemma}

\begin{proof}
Since $f$ is degree two polynomial in $x$, we have
$$\frac{d^2f}{dx^2}\,,\,(\frac{df}{dx})^2\in \mathsf{R}_f\,.$$
Then the proof of Lemma follows from simple induction argument.
\end{proof}

\section{Local $\PP^1$}\label{PP1}
\subsection{Overview}In this section, we prove Conjecture \ref{MC} for the case $n=1$.
Recall the $I$-function for $K\PP^1$,
 \begin{align} 
I_1(q) = \sum _{d=0}^{\infty}  \frac{ \prod _{k=0}^{2d-1}  (-2H - kz)}{\prod^1_{i=0}\prod _{k=1}^d (H-\lambda_i+kz)}q^d . \end{align}
The function $I_1$ satisfies following Picard-Fuchs equation
\begin{align}
\label{PFlp1}\Big((M-\lambda_0)(M-\lambda_1)-2qM(2M+z)\Big) 
I_1=0\,.
\end{align}
Recall the notation used in above equation,
\begin{align*}
    \mathsf{D}=q \frac{d}{dq}\,,\,\,\,M=H+z\mathsf{D}\,.
\end{align*}
The restriction $I_1|_{H=\lambda_i}$ admits following asymptotic form
\begin{align}
    I_1|_{H=\lambda_i}=e^{\mu_i/z}\left( R_{0,i}+R_{1,i} z+R_{2,i} z^2+\ldots\right)
\end{align}
with series $\mu_i, R_{k,i}\in \CC(\lambda_0,\lambda_1)[[q]]$. The series $\mu_i$ and $R_{k,i}$ are found by solving differential equations obtained from the coefficient of $z^k$ in \eqref{PFlp1}. For example, we have for $i=0,1$,
\begin{align*}
    &\lambda_i+\mathsf{D}\mu_i=L_i\,,\\
    &R_{0,i}=\Big(\frac{\lambda_i\prod_{j \ne i}(\lambda_i-\lambda_j)}{f_1(L_i)}\Big)^{\frac{1}{2}}\,,\\
\end{align*}
\begin{multline*}
     R_{1,i}=\Big(\frac{\lambda_i\prod_{j \ne i}(\lambda_i-\lambda_j)}{f(L_i)}\Big)^{\frac{1}{2}} \cdot\\
     \Big(\frac{-16 s_1^2 s_2^2 + 
   88 s_2^3 + (27 s_1^3 s_2 - 132 s_1 s_2^2) L_i + (-12 s_1^4 + 
    54 s_1^2 s_2) L_i^2}{24 s_1 (L_i s_1 - 2 s_2)^3}\\+\frac{12 \lambda_i^2 - 9 \lambda_i \lambda_{i+1} + \lambda_{i+1}^2}{24 (\lambda_i^3 - \lambda_i \lambda_{i+1}^2)}\Big)\,.
\end{multline*}
Here $s_1=\lambda_0+\lambda_1$ and $s_2=\lambda_0 \lambda_1$. In the above expression of $R_{1,i}$, we used the convention $\lambda_2=\lambda_0$.

\subsection{Proof of Conjecture \ref{MC}.}
We introduce new differential operator $\mathsf{D}_i$ defined by for $i=0,1$,
$$\mathsf{D}_i=(\mathsf{D}L_i)^{-1}\mathsf{D}\,. $$
By definition, $\mathsf{D}_i$ acts on rational functions in $L_i$ as the ordinary derivation with respect to $L_i$.
If we use following normalizations,
\begin{align*}
    R_{k,i}=f_1(L_i)^{-\frac{1}{2}} \Phi_{k,i}
\end{align*}
the Picard-Fuchs equation \eqref{PFlp2} yields the following differential equations,
\begin{align}\label{DFlp1}
    \mathsf{D}_i\Phi_{p,i}-A_{00,i}\Phi_{p-1,i}-A_{01,i}\mathsf{D}_i\Phi_{p-1,i}-A_{02,i}\mathsf{D}_i^2\Phi_{p-1,i}=0\,,
\end{align}
with 
\begin{align*}
    A_{00,i}&=\frac{-s_1^2 s_2^2 + (-s_1^3 s_2 + 8 s_1 s_2^2) L_i + (2 s_1^4 - 9 s_1^2 s_2) L_i^2}{4 (L_i s_1 - 2 s_2)^4}\,,\\
    A_{01,i}&=\frac{2 s_1 s_2^2 + (-s_1^2 s_2 - 8 s_2^2) L_i + (-s_1^3 + 10 s_1 s_2) L_i^2 - s_1^2 L_i^3}{2 (L_i s_1 - 2 s_2)^3}\,,\\
    A_{02,i}&=\frac{s_2^2 - 2 (s_1 s_2) L_i + (s_1^2 + s_2) L_i^2 - s_1 L_i^3}{(L_i s_1 - 2 s_2)^2}\,.
\end{align*}
Here $s_k$ is the $k$-th elementary symmetric functions in $\lambda_0,\lambda_1$. 
Since the differential equations \eqref{DFlp1} satisfy the condition \eqref{D1C}, we conclude the differential equations \eqref{DFlp1} is admissible.

\subsection{Gomov-Witten series.}
By the result of previous subsection, we obtain the following result which verifies the prediction (A) in Section \ref{LPPN}.
\begin{Thm}\label{GWPP1}For the Gromov-Witten series of $K\PP^1$, we have 
 $$\mathcal{F}^{\mathsf{GW},1}_g(Q(q)) \in \mathds{G}_1\,,$$
where $Q(q)$ is the mirror map of $K\PP^1$ defined by
$$Q(q):=q\cdot\text{exp}\Big(2\sum_{d=1}^\infty\frac{(2d-1)!}{(d!)^2}q^d\Big)\,.$$
\end{Thm}
\noindent Theorem \ref{GWPP1} follows from the argument in \cite{LP}. The prediction (B) in Section \ref{LPPN} is trivial statement for $K\PP^1$.

\section{Local $\PP^2$}\label{PP2}
\subsection{Overview}In this section, we prove Conjecture \ref{MC} for the case $n=2$ with following specializations,
\begin{align}\label{spl2}
    (\lambda_0 \lambda_1+\lambda_1 \lambda_2+\lambda_2\lambda_0)^2-3\lambda_0\lambda_1\lambda_2(\lambda_0+\lambda_1+\lambda_2)=0\,.
\end{align}
The predictions (A) and (B) in Section \ref{LPPN} are studied in \cite{HL} based on the result of this section.
For the rest of the section, the specialization \eqref{spl2} will be imposed. Recall the $I$-function for $K\PP^2$.
 \begin{align}\label{I_Hyper_P} 
I_2(q) = \sum _{d=0}^{\infty}  \frac{ \prod _{k=0}^{3d-1}  (-3H - kz)}{\prod^2_{i=0}\prod _{k=1}^d (H-\lambda_i+kz)}q^d . \end{align}

The function $I_2$ satisfies following Picard-Fuchs equation
\begin{align}
\label{PFlp2}\Big((M-\lambda_0)(M-\lambda_1)(M-\lambda_2)+3qM(3M+z)(3M+2z)\Big) 
I_2=0
\end{align}
Recall the notation used in above equation,
\begin{align*}
    \mathsf{D}=q \frac{d}{dq}\,,\,\,\,M=H+z\mathsf{D}\,.
\end{align*}

The restriction $I_2|_{H=\lambda_i}$ admits following asymptotic form
\begin{align}
    I_2|_{H=\lambda_i}=e^{\mu_i/z}\left( R_{0,i}+R_{1,i} z+R_{2,i} z^2+\ldots\right)
\end{align}
with series $\mu_i, R_{k,i}\in \CC(\lambda_0,\lambda_1,\lambda_2)[[q]]$. The series $\mu_i$ and $R_{k,i}$ are found by solving differential equations obtained from the coefficient of $z^k$ in \eqref{PFlp2}. For example,
\begin{align*}
    \lambda_i+\mathsf{D}\mu_i=L_i\,,\\
     R_{0,i}=\Big(\frac{\lambda_i\prod_{j \ne i}(\lambda_i-\lambda_j)}{f_2(L_i)}\Big)^{\frac{1}{2}}\,.\\
\end{align*}

\subsection{Proof of Conjecture \ref{MC}.}
We introduce new differential operator $\mathsf{D}_i$ defined by
$$\mathsf{D}_i=(\mathsf{D}L_i)^{-1}\mathsf{D}\,. $$

If we use following normalizations,
\begin{align*}
    R_{k,i}=f_2(L_i)^{-\frac{1}{2}} \Phi_{k,i}
\end{align*}
the Picard-Fuchs equation \eqref{PFlp2} yields the following differential equations,
\begin{multline}\label{DFlp2}
    \mathsf{D}_i\Phi_{p,i}-A_{00,i}\Phi_{p-1,i}-A_{01,i}\mathsf{D}_i\Phi_{p-1,i}-A_{02,i}\mathsf{D}_i^2\Phi_{p-1,i}\\-A_{10,i}\Phi_{p-2,i}-A_{11,i}\mathsf{D}_i\Phi_{p-2,i}-A_{12,i}\mathsf{D}_i^2\Phi_{p-2,i}-A_{13,i}\mathsf{D}_i^3\Phi_{p-2,i}=0\,,
\end{multline}
with $A_{jl,i}\in\CC(\lambda_0,\lambda_1,\lambda_2)[L_i,f_2(L_i)^{-1}]$.
We give the exact values of $A_{jl,i}$ for reader's convinience.
\begin{multline*}
    A_{00,i}=\frac{s_1}{9(s_1 L_i-s_2)^5}\Big(s_1 s_2^3+(-4s_1^2s_2^2+3s_2^3)L_i\\+(-s_1^3 s_2+12s1 s_2^2)L_i^2+(11s_1^4-36s_1^2 s_2)L_i^3\Big)\,,
\end{multline*}

\begin{multline*}
    A_{01,i}=\frac{-s_1}{3(s_1 L_i-s_2)^4}\Big(s_2^3-4(s_1 s_2^2)L_i+(3s_1^2s_2+9s_2^2)L_i^2\\+(3s_1^3-21s_1 s_2)L_i^3+3s_1^2 L_i^4   \Big)\,,
\end{multline*}

\begin{multline*}
    A_{02,i}=\frac{-1}{3(s_1 L_i-s_2)^3}\Big( s_2^3-5(s_1 s_2^2)L_i+9s_1^2 s_2 L_i^2+(-6s_1^3-3s_1 s_2)L_i^3\\+6s_1^2 L_i^4 \Big)\,,
\end{multline*}

\begin{multline*}
    A_{10,i}=\frac{s_1^2 L_i}{27(s_1 L_i-s_2)^9}\Big( (8 s_1^2 s_2^5 - 
   21 s_2^6) + (-48 s_1^3 s_2^4 + 126 s_1 s_2^5) L_i + (120 s_1^4 s_2^3 \\- 
    315 s_1^2 s_2^4) L_i^2 + (-124 s_1^5 s_2^2 + 264 s_1^3 s_2^3 + 
    144 s_1 s_2^4) L_i^3 \\+ (12 s_1^6 s_2 + 153 s_1^4 s_2^2 - 
    432 s_1^2 s_2^3) L_i^4 + (60 s_1^7 - 342 s_1^5 s_2 \\+ 
    432 s_1^3 s_2^2) L_i^5 + (-33 s_1^6 + 108 s_1^4 s_2) L_i^6    \Big)\,,
\end{multline*}

\begin{multline*}
    A_{11,i}=\frac{-s_1 L_i}{27(s_1 L_i-s_2)^8}\Big( (8 s_1^2 s_2^5 - 
   21 s_2^6) + (-48 s_1^3 s_2^4 + 126 s_1 s_2^5) L_i\\ + (120 s_1^4 s_2^3 - 
    315 s_1^2 s_2^4) L_i^2 + (-124 s_1^5 s_2^2 + 264 s_1^3 s_2^3 + 
    144 s_1 s_2^4) L_i^3 \\+ (12 s_1^6 s_2 + 153 s_1^4 s_2^2 - 
    432 s_1^2 s_2^3) L_i^4 + (60 s_1^7 - 342 s_1^5 s_2\\ + 
    432 s_1^3 s_2^2) L_i^5 + (-33 s_1^6 + 108 s_1^4 s_2) L_i^6 \Big)\,
\end{multline*}

\begin{multline*}
    A_{12,i}=\frac{s_1}{9(s_1 L_i-s_2)^7}\Big(-s_2^6 + 9 s_1 s_2^5 L_i + (-32 s_1^2 s_2^4 - 9 s_2^5) L_i^2 \\+ (57 s_1^3 s_2^3 + 
    60 s_1 s_2^4) L_i^3 + (-48 s_1^4 s_2^2 - 
    171 s_1^2 s_2^3) L_i^4\\ + (9 s_1^5 s_2 + 237 s_1^3 s_2^2 + 
    27 s_1 s_2^3) L_i^5 + (9 s_1^6 - 144 s_1^4 s_2 - 
    90 s_1^2 s_2^2) L_i^6 \\+ (9 s_1^5 + 108 s_1^3 s_2) L_i^7 - 18 s_1^4 L_i^8\Big)\,,
\end{multline*}

\begin{multline*}
    A_{13,i}=-\frac{(3 L_i^2 s_1^2 - 3 L_i s_1 s_2 + s_2^2) (-3 L_i^3 s_1 + 3 L_i^2 s_1^2 - 3 L_i s_1 s_2 + 
   s_2^2)^2}{27 (s_1 L_i - s_2)^6}\,.
\end{multline*}
Here $s_k$ is the $k$-th elementary symmetric functions in $\lambda_0,\lambda_1,\lambda_2$. 
Since the differential equations \eqref{DFlp2} satisfy the condition \eqref{D1C}, we conclude that the differential equations \eqref{DFlp2} is admissible.


\begin{thebibliography}{99}



\bibitem{ASYZ} M. Alim, E. Scheidegger, S.-T. Yau, J. Zhou, {\em Special polynomial rings, quasi modular forms and duality of topological strings}, Adv. Theor. Math. Phys. {\bf 18} (2014), 401--467.



\bibitem{BCOV} M. Bershadsky, S. Cecotti, H. Ooguri, and C. Vafa, {\em Holomorphic anomalies in topological field theories}, Nucl. Phys. B{\bf 405} (1993), 279--304.




\bibitem {CKg0}  I. Ciocan-Fontanine and B. Kim, {\em Wall-crossing in genus zero quasimap theory and mirror maps,}   Algebr. Geom. {\bf 1 } (2014), 400--448.










2051-2102.



























\bibitem{HL} H. Lho, {\em Equivariant holomorphic anomaly equation}, in preparation.

\bibitem{LP} H. Lho and R. Pandharipande,
{\em Stable quotients and the holomorphic anomaly equation}, Adv. in Math. {\bf 332} (2018), 349--402.

\bibitem{LP2} H. Lho and R. Pandharipande, {\em Holomorphic anomaly equations for the formal quintic}, arXiv:1803.01409.








\bibitem{ObPix} G. Oberdieck and A. Pixton, {\em Gromov-Witten theory of elliptic fibrations: Jacobi forms and holomorphic anomaly equations}, arXiv:1709.01481.


\bibitem{ZaZi} D. Zagier and A. Zinger, {\em Some properties of hypergeometric series associated with mirror symmetry} in {\em Modular Forms and String Duality}, 163-177, Fields Inst. Commun. {\bf 54}, AMS 2008.







\end{thebibliography}
\end{document}